\documentclass[12pt,showkeys]{amsart}
\usepackage{color}
\usepackage{graphicx}
\usepackage{subfigure}
\usepackage{amssymb}
\usepackage{amsmath}
\usepackage{multirow}
\usepackage{enumerate}
\usepackage{epstopdf}
\usepackage{cite,stmaryrd,txfonts}
\usepackage{tikz}

\input{undertilde}
\usetikzlibrary{snakes}

\usepackage{verbatim}

\usepackage{epic}

\usepackage{empheq}

\usepackage{ulem}

\newcommand{\bs}{\boldsymbol}
\newcommand{\mrm}{\mathrm}

\def\cequiv{\raisebox{-1.5mm}{$\;\stackrel{\raisebox{-3.9mm}{=}}{{\sim}}\;$}}

\def\mtequiv{\raisebox{-1.5mm}{$\;\stackrel{\raisebox{-3.9mm}{[3]}}{=}\;$}}

\def\utau{\undertilde{\tau}}
\def\uphi{\undertilde{\varphi}}
\def\upsi{\undertilde{\psi}}

\def\uG{\undertilde{G}}

\def\uH{\undertilde{H}}

\def\uS{\undertilde{S}}

\def\uP{\undertilde{P}}

\def\uw{\undertilde{w}}
\def\uv{\undertilde{v}}

\def\uL{\undertilde{L}}

\def\rot{{\rm rot}}
\def\curl{{\rm curl}}
\def\dv{{\rm div}}

\def\llt{\llbracket}
\def\rrt{\rrbracket}

\def\vgm12{\bs{V}^{1+,2}_{\gamma,M}}

\newtheorem{theorem}{Theorem}
\newtheorem{remark}[theorem]{Remark}

\newtheorem{lemma}[theorem]{Lemma}

\newcounter{mnote}
\let\oldmarginpar\marginpar
\renewcommand\marginpar[1]{\-\oldmarginpar[\raggedleft\footnotesize #1]%
  {\raggedright\footnotesize #1}}

\begin{document}

\title{On optimal finite element schemes for biharmonic equation}
\author{Shuo Zhang}
\address{LSEC, Institute of Computational Mathematics and Scientific/Engineering Computing, Academy of Mathematics and System Sciences, Chinese Academy of Sciences, Beijing 100190, People's Republic of China}
\email{szhang@lsec.cc.ac.cn}
\thanks{The author is supported partially by the National Natural Science Foundation of China with Grant No. 11471026 and National Centre for Mathematics and Interdisciplinary Sciences, Chinese Academy of Sciences.}



\begin{abstract} 
In this paper, two nonconforming finite element schemes that use piecewise cubic and piecewise quartic polynomials respectively are constructed for the planar biharmonic equation with optimal convergence rates on general shape-regular triangulations. Therefore, it is proved that optimal finite element schemes of arbitrary order for planar biharmonic equation can be constructed on general shape-regular triangulations. 
\end{abstract}

\maketitle

\section{Introduction}

In the study of qualitative and numerical analysis of partial differential equations and, in general, of approximation theory, we are often interested in the approximation of functions in Sobolev spaces by piecewise polynomials defined on a partition of the domain. In order for simpler interior structure, lower-degree polynomials are often expected to be used with respect to the same convergence rate. When finite element spaces that consist of polynomials of ${\bf k}$-th degree are used for discretizing $H^{\bf m}$ elliptic problems, the convergence rate in energy norm can generally not be expected higher than $\mathcal{O}(h^{k+1-m})$ for general grids, and finite element schemes that possess convergence rate $\mathcal{O}(h^{k+1-m})$ are called {\bf optimal}. It is of theoretical and practical interests to construct {optimal} finite element schemes, and this paper is devoted to this problem. 
~\\

Optimal schemes have been well studied for lowest-order cases. On one hand, for the lowest-order ($H^1$) elliptic problems,  the standard Lagrangian conforming elements can provide optimal approximation on simplicial grids of arbitrary dimension. Optimal nonconforming element spaces of $k$-th degrees are also constructed; we refer to, e.g., \cite{Crouzeix.M;Raviart.P1973},  \cite{Fortin.M;Soulie.M1983} and \cite{Crouzeix.M;Falk.R1989} for the cases $k=1$, $k=2$ and $k=3$, respectively, and to \cite{Baran.A;Stoyan.G2007} for general $k$. On the other hand, for general $H^m$ elliptic problems,  minimal-degree approximations have been studied. Specifically, when the subdivision consists of simplexes, a systematic family of nonconforming finite elements has been proposed by  \cite{Wang.M;Xu.J2013} for $H^m$ elliptic partial differential equations in $R^n$ for any $n\geqslant m$ with polynomials with degree $m$. Known as Wang-Xu family or Morley-Wang-Xu family, the elements have been playing bigger and bigger role in numerical analysis. The elements are constructed based on the perfect matching between the dimension of $m$-th degree polynomials and the dimension of $(n-k)$-subsimplexes with $1\leqslant k\leqslant m$. The generalization to the cases $n<m$ is attracting more and more research interests; c.f., e.g., \cite{Wu.S;Xu.J2017}. The minimal conforming element spaces have been constructed for $\mathbb{R}^n$ rectangle grids by \cite{Hu.J;Zhang.Sy2015}, where $Q_m$ polynomials are used for $H^m$ problems; for these spaces, composite grids are used. Besides,  constructions of finite element functions that does not depend on a cell-by-cell definition can be found in \cite{Park.C;Sheen.D2003,Hu.J;Shi.Z2005,Zhang.S2017axv}, where minimal-degree finite element spaces are defined on general quadrilateral grids for $H^1$ and $H^2$ problems. For these schemes,  the {\it finite element spaces} are defined globally and do not necessarily correspond to a {\it finite element} defined by Ciarlet's triple \cite{Ciarlet.P1978}. 
~\\

Higher-order cases are more complicated, even for the planar biharmonic problem. It is known that with polynomials of degrees $k\geqslant 5$, spaces of $\mathcal{C}^1$ continuous piecewise polynomials can be constructed with local basis, and they provide optimal approximations to the $H^2$ functions with sufficient smoothness\cite{Argyris.J;Fried.I;Scharpf.D1968,Zenisek.A1974,Morgan.J;Scott.R1975,Zenisek.A1970,deBoor.C;Hollig.K1988}. With polynomials of degrees $2\leqslant k\leqslant 4$, spaces of $\mathcal{C}^1$ continuous piecewise polynomials can be proved to provide optimal approximation when the triangulation is of some special structures, such as the Powell-Sabin triangulation\cite{Powell.M;Sabin.M1977}, Hsieh-Clough-Tocher triangulation\cite{Clough.R;Tocher.J1965}, Sander-Veubeke triangulation\cite{Sander.G1964,deVeubeke.BF1968}. The restrictions on the grids can be weakened, but are generally needed on at least part of the triangulation\cite{Nurnberger.G;Zeilfelder.F2004,Nurnberger.G;Schumaker.L;Zeilfelder.F2004,Chui.C;Hecklin.G;Nurnberger.G;Zeilfelder.F2008}. On general triangulations, as is proved by \cite{deBoor.C;Jia.R1993}, and illustrated by a counter example on a very regular triangulation\cite{Boor.C;DeVore.R1983,Boor.C;Hollig.K1983}, optimal approximation can not be expected with $\mathcal{C}^1$ continuous piecewise polynomials of degree $k<5$. It is illustrated in \cite{Alfeld.P;Piper.B;Schumaker.L1987} that all the basis functions can not be determined locally on general grids. In contrast, nonconforming finite element methodology works optimally for $k=2$. The Morley element \cite{Morley.L1968} uses piecewise quadratic polynomials to discretize biharmonic equation, and the convergence rate $\mathcal{O}(h)$ is proved. However,  to the author's knowledge, optimal finite element schemes, conforming or nonconforming, for planar biharmonic equation with $\mathcal{O}(h^2)$ or $\mathcal{O}(h^3)$ convergence rate, namely $m=2$ and $k=3,4$, are still absent. For biharmonic problem in higher dimensions and higher order problems, a bigger absence than a narrow gap can be expected.  
~\\

In this paper, we present two nonconforming finite element spaces $B_h^3$ and $B_h^4$, which consist of piecewise cubic and piecewise quartic polynomials, and two schemes which approximate the planar biharmonic equation with $\mathcal{O}(h^2)$ and $\mathcal{O}(h^3)$ convergence rate in energy norm on general shape-regular triangulations, respectively, and the convergence analysis, implementation and optimal solvers are discussed. As can be imagined, to control the consistency error, sufficient restrictions on interfacial continuity have to be imposed across the edges of the cells, and this way are the two finite element spaces defined. However, the constraints on the continuity are overdetermined versus local shape functions, and therefore, the globally-defined finite element spaces can not be corresponding to finite elements defined as Ciarlet's triple.  Difficulties arise thus in both theoretical analysis and practical implementation, especially on constructing local basis functions or even counting the dimension of the space which is crucial on implementation and on error estimation by interpolation. Indirect ways are adopted in this paper to avoid the difficulties, and the construction and utilization of discretized Stokes complexes are the main ingredients in the construction of the schemes. 
~\\

Discretized Stokes complexes are finite element analogues of the Stokes complex which reads 
\begin{equation}
\begin{array}{ccccccccc}
0 & \longrightarrow & H^2_0 & \xrightarrow{\bs{\mrm{\nabla}}} & (H^1_0)^2 & \xrightarrow{\mrm{rot}} & L^2_0  & \longrightarrow & 0.
\end{array}
\end{equation}
Discretized Stokes complexes are playing important roles in constructing stable finite element pairs for Stokes problem \cite{Falk.R;Neilan.M2013,John.V;Linke.A;Merdon.C;Neilan.M;Rebholz.L2017}. In this paper, the ability of the Stokes complex to figure out the relevance between the $H^2$ space and a constrained $H^1$ space is emphasized. Specifically, two discretized Stokes complexes are constructed, and the schemes for biharmonic equation are connected to that for rotated Stokes problem. This way, the original error estimation and implementation problems can be transformed to those for relevant rotated Stokes problems to solve.  Moreover, as  $B_h^3$ and $B_h^4$ are defined globally, the dimensions of the spaces are not trivially known, and the standard dimension counting technique can not be used directly for the proof of the exactness between the spaces. These two discrete Stokes complexes are constructed in an indirect way: we firstly construct two auxiliary discrete Stokes complexes, and then reduce them to desired ones, respectively. Finally, the theoretical analysis and implementation of the designed finite element schemes are done in a friendly way. Particularly, the scheme can be implemented by solving two Poisson systems and one Stokes systems, and thus optimal solvers for the generated discretizations can also be constructed under the framework of fast auxiliary space preconditioning (FASP) and by the aid of optimal Poisson solvers; see discussions in, e.g., \cite{Xu.J1996as,Hiptmair.R;Xu.J2007,Xu.J2010icm,Zhang.S;Xu.J2014,Ruesten.T;Winther.R1992,Grasedyck.L;Wang.L;Xu.J2016,Feng.C;Zhang.S2016} for relevant details.
~\\

We remark that, though indirect ways are utilized in the present paper, the finite element spaces ($B_h^3$ for example) may still admit a set of locally supported basis functions which may be used for interpolation-based technique and data-fitting-oriented problems. We also remark that, the discontinuous Galerkin (DG) method for biharmonic equation has been studied, and various optimal schemes are constructed; c.f., e.g., \cite{Brenner.S;Sung.L2005,Engel.G;Garikipati.K;Hughes.TJR;Larson.MG;Mazzei.L;Taylor.RL2002,Mozolevski.I;Suli.E;Bosing.P2007,Georgoulis.E;Houston.P2008}. In contrast to existing DG schemes, the nonconforming finite element methods presented in this paper can hint new kinds of DG methods. Moreover, there can be possibly asymptotic connections between the nonconforming finite element schemes and the induced DG schemes; the asymptotic connections can be helpful to studying the structure of the DG systems. 
~\\

The remaining of the paper is organized as follows. In the remaining part of this section, the notation used in this paper are introduced. In Section \ref{sec:fes}, some auxiliary finite elements are introduced. In Sections \ref{sec:cubic} and \ref{sec:quartic}, optimal cubic and quartic element spaces for the biharmonic equation are introduced with structural properties. Their implementation for solving biharmonic equation and convergence analysis are given. Finally, in Section \ref{sec:conc}, some concluding remarks are given. 

\subsection{Notations}

Through this paper, we use $\Omega$ for a  simply-connected polygonal domain.  We use $\nabla$, $\curl$, $\dv$, $\rot$ and $\nabla^2$ for the gradient operator, curl operator, divergence operator, rot operator and Hessian operator, respectively. As usual, we use $H^2(\Omega)$, $H^2_0(\Omega)$, $H^1(\Omega)$, $H^1_0(\Omega)$, $H(\rot,\Omega)$, $H_0(\rot,\Omega)$ and $L^2(\Omega)$ for certain Sobolev spaces, and specifically, denote  $\displaystyle L^2_0(\Omega):=\{w\in L^2(\Omega):\int_\Omega w dx=0\}$, $\undertilde{H}{}^1_0(\Omega):=(H^1_0(\Omega))^2$, $\mathring{H}_0(\rot,\Omega):=\{\utau\in H_0(\rot,\Omega):\rot\utau=0\}$, $\mathring{\uH}{}^1_0(\Omega):=\{\upsi\in \uH{}^1_0(\Omega):\rot\upsi=0\}$, and $\uH{}^{-1}(\Omega)$ and $H^{-1}(\Omega)$ the dual spaces of $\uH{}^1_0(\Omega)$ and $H^1_0(\Omega)$, respectively.  We use $``\undertilde{~}"$ for vector valued quantities in the present paper, and $\uv{}^1$ and $\uv{}^2$ for the two components of the function $\uv$. We use $(\cdot,\cdot)$ for $L^2$ inner product and $\langle\cdot,\cdot\rangle$ for the duality between a space and its dual. Without ambiguity, we use the same notation $\langle\cdot,\cdot\rangle$ for different dualities, and it can occasionally be treated as $L^2$ inner product for certain functions. We use the subscript $``\cdot_h"$ to denote the dependence on triangulation. Particularly, an operator with the subscript $``\cdot_h"$ implies the operation done cell by cell. Finally, $\lesssim$, $\gtrsim$, and $\cequiv$ respectively denote $\leqslant$, $\geqslant$, and $=$ up to a constant. The hidden constants depend on the domain, and, when triangulation is involved, they also depend on the shape-regularity of the triangulation, but they do not depend on $h$ or any other mesh parameter.

Let $\mathcal{T}_h$ be a shape-regular triangular subdivision of $\Omega$ with mesh size $h$, such that $\overline\Omega=\cup_{T\in\mathcal{T}_h}\overline T$. Denote by $\mathcal{E}_h$, $\mathcal{E}_h^i$, $\mathcal{E}_h^b$, $\mathcal{X}_h$, $\mathcal{X}_h^i$, and $\mathcal{X}_h^b$ the set of edges, interior edges, boundary edges, vertices, interior vertices and boundary vertices, respectively. For any edge $e\in\mathcal{E}_h$, denote by $\mathbf{n}_e$ and $\mathbf{t}_e$ the unit normal and tangential vectors of $e$, respectively, and denote by $\llbracket\cdot\rrbracket_e$ the jump of a given function across $e$; if particularly $e\in\mathcal{E}_h^b$, $\llbracket\cdot\rrbracket_e$ stands for the evaluation of the function on $e$. The subscript ${\cdot}_e$ can be dropped when there is no ambiguity brought in. 

Denote 
$$
\mathcal{X}_h^{b,+1}:=\{a\in\mathcal{X}_h^i,\ a\ \mbox{is\ connected\ to}\ \mathcal{X}_h^b\ by\ e\in\mathcal{E}_h^i\},\ \mbox{and}\ \ \mathcal{X}_h^{i,-1}:=\mathcal{X}_h^i\setminus\mathcal{X}_h^{b,+1};
$$ 
further, denote with $\mathcal{X}^{i,-(k-1)}_h\neq\emptyset$,
$$
\mathcal{X}_h^{b,+k}:=\{a\in\mathcal{X}_h^{i,-(k-1)},\ a\ \mbox{is\ connected\ to}\ \mathcal{X}_h^{b,+(k-1)}\ by\ e\in\mathcal{E}_h^i\}, \ \mbox{and}\ \ \mathcal{X}_h^{i,-k}:=\mathcal{X}_h^{i,-(k-1)}\setminus\mathcal{X}_h^{b,+k}.
$$
The smallest $k$ such that $\mathcal{X}_h^{i,-(k-1)}=\mathcal{X}_h^{b,+k}$ is called the number of levels of the triangulation. 

For $T$ a triangle, we use $P_k(T)$ for the set of polynomials on $K$ of degrees not higher than $k$. Similarly is $P_k(e)$ defined on an edge $e$. Define $\uP{}_k(T)=P_k(T)^2$ and similarly is $\uP{}_k(e)$ is defined.  Denote by $a_i$ and $e_i$ vertices and opposite edges of $K$, $i=1,2,3$.  The barycentre coordinates are denoted as usual by $\lambda_i$, $i=1,2,3$. Besides, denote $\Lambda=\lambda_1\lambda_2\lambda_3$. The indices of $a_i$, $e_i$ and $\lambda_i$ are labelled mod 3, namely $a_{i+3}=a_i$, and so forth. Correspondingly, we use the notation $``i\mtequiv j"$ to denote $``i\,({\rm mod}\,3)=j\,({\rm mod}\,3)"$. Denote basic finite element spaces by
\begin{itemize}
\item $\mathcal{L}^k_h:=\{w\in H^1(\Omega):w|_T\in P_k(T),\ \forall\,T\in\mathcal{T}_h\}$, $\mathcal{L}^k_{h0}:=\mathcal{L}^k\cap H^1_0(\Omega)$, $k\geqslant 1$;
\item $\mathbb{P}^{k}_h:=\{w\in L^2(\Omega):w|_T\in P_{k}(T)\}
$, $\mathbb{P}^{k}_{h0}:=\mathbb{P}^k_h\cap L^2_0(\Omega)$, $k\geqslant 0$.
\end{itemize}

\section{Some finite elements and associated finite element spaces}
\label{sec:fes}

In this section, we present some finite elements for further discussion in the sequel sections. Some of them are known, while most of them are not. 
\subsection{On stable Stokes finite element pairs with discontinuous pressure}

For $k\geqslant 1$, define
\begin{itemize}
\item $\uS{}_h^k:=\{\uv\in\uH^1(\Omega): \uv|_T\in \uP{}_k(T),\ T\in\mathcal{T}_h\}$;
\item $\uS{}^k_{h0}:=\uS{}^k_h\cap\uH{}^1_0(\Omega)$;
\item $\uG{}_h^k:=\{\uv\in\uL^2(\Omega):\int_ep_e\llt\uv^j\rrt=0,\ \forall\, p_e\in P_{k-1}(e),\ \forall\, e\in \mathcal{E}_h^i,\ j=1,2\};$
\item $\uG{}^k_{h0}:=\{\uv\in\uG{}^k_h:\int_ep_e\uv^j=0,\ \forall\,e\in \mathcal{E}_h^b\ \mbox{and}\ p_e\in P_{k-1}(e),\ j=1,2\}.$
\end{itemize}
Note that the moment-continuity of $\uG{}^k_h$ is equivalent to continuity on certain Gauss-Legendre points.  Denote $\undertilde{\mathcal{B}}{}^2_{h0}:=\{\undertilde{\phi}{}_h:(\undertilde{\phi}{}_h|_T)^j\in {\rm span}\{(\lambda_1^2+\lambda_2^2+\lambda_3^2)-2/3\}, \ j=1,2,\ \forall\,T\in\mathcal{T}\}$. Namely $\undertilde{\mathcal{B}}{}^2_{h0}$ is a second-degree Gauss-Legendre bubble space. Then $\uG{}^2_{h0}=\uS{}^2_{h0}\oplus \undertilde{\mathcal{B}}{}^2_{h0}$\cite{Fortin.M;Soulie.M1983}.  The decomposition can be generalized to even $k$\cite{Baran.A;Stoyan.G2007}. These spaces can be used for  Stokes problem. 

\begin{lemma}\label{lem:P4P3}
\cite{Guzman.J;Scott.R2018} (Stability of the Scott-Vogelius pair)\footnote{The Stability is proved firstly in \cite{Scott.L;Vogelius.M1985} where it is assumed the triangulation does not contain singular points.}
For $k\geqslant 4$, then there exists a generic constant $C$ depending on the domain and the regularity of the grid, such that 
\begin{equation}\label{eq:inf-supSV}
\sup_{\uv{}_h\in \uS{}^k_{h0},\|\nabla_h\uv{}_h\|_{0,\Omega}=1}(\dv \uv{}_h,q_h)\geqslant C\|q_h\|_{0,\Omega},\quad\forall\, q_h\in \mathbb{P}^{k-1}_{h0}. 
\end{equation}
\end{lemma}

\begin{lemma}\label{lem:pGL}
For $k=1,2$ or $k\geqslant 4$, there exists a generic constant $C$ depending on the domain and the regularity of the grid, such that 
\begin{equation}\label{eq:inf-supGL}
\sup_{\uv{}_h\in \uG{}^k_{h0},\|\nabla_h\uv{}_h\|_{0,\Omega}=1}(\dv_h\uv{}_h,q_h)\geqslant C\|q_h\|_{0,\Omega},\ \ \forall\,q_h\in\mathbb{P}^{k-1}_{h0}.
\end{equation}
\end{lemma}
\begin{proof}
For $k=1$, the pair is known as Crouzeix-Raviart pair and the lemma can be founded proved in \cite{Crouzeix.M;Raviart.P1973}; for $k=2$, it is known as Fortin-Soulie pair and is studied in \cite{Fortin.M;Soulie.M1983}; for $k\geqslant 4$, it is true by Lemma \ref{lem:P4P3} as $\uS{}^k_{h0}\subset \uG{}^k_{h0}$.
\end{proof}

\begin{remark}\label{rem:cfinfsup}
The case $k=3$ corresponds to the Crouzeix-Falk pair and was studied in \cite{Crouzeix.M;Falk.R1989}. In that paper, they proved that the pair $\uG{}^3_{h0}-\mathbb{P}^2_{h0}$ is stable ``for most reasonable meshes". Moreover, they presented a conjecture that the pair is stable ``for any triangulation of a convex polygon satisfying the minimal angle condition and containing an interior vertex". Recently, some triangulations where $\uS{}^3_{h0}-\mathbb{P}_{h0}^2$ is stable or at least $\dv\uS{}^3_{h0}=\mathbb{P}_{h0}^2$ are listed in \cite{guzman2017cubic}, and $\uG{}^3_{h0}-\mathbb{P}_{h0}^2$ is stable or $\dv_h\uG{}^3_{h0}=\mathbb{P}_{h0}^2$ on these triangulations. 

\end{remark}
\begin{remark}
By the symmetry between the two components of $\uH{}^1(\Omega)$, Lemmas \ref{lem:P4P3} and \ref{lem:pGL} remain true when ``$\dv$" and ``$\dv{}_h$" are replaced by ``$\rot$" and ``$\rot_h$", respectively. 
\end{remark}

\subsection{New cubic and quartic elements}

\subsubsection{A new cubic element {\bf FE${}_{\rm nsc}$}}

The finite element {\bf FE${}_{\rm nsc}$} is defined by $(T, P_T, D_T)$ with

\fbox{
\begin{minipage}{0.95\textwidth}

\begin{itemize}
\item $T$ is a triangle;
\item $P_T=P_3(T)$;
\item the components of $D_T$ for $v\in H^2(T)$ are:
$$
d_i(v)=v(a_i);\ d_{i+3}(v)=\fint_{e_i}v;\ d_{i+6}(v)=\fint_{T}(\lambda_i^2\lambda_{i+1}-\lambda_i\lambda_{i+1}^2)v;\ i=1,2,3; d_{10}(v)=\fint_T v.
$$
\end{itemize}
\end{minipage}
}
\begin{lemma}
The finite element {\bf FE${}_{\rm nsc}$} is well defined. 
\end{lemma}
\begin{proof}

It is well known that
$$
d_j(\lambda_i^2)=\delta_{ij},\quad d_j(\lambda_i\lambda_{i+1})=0,
\quad d_{j+3}(\lambda_i\lambda_{i+1})=\delta_{ij},\quad d_{j+3}(\lambda_i^2)=0, \ \ i,j=1:3.
$$
To show the unisolvence of the space, it suffices to show there exist a subset of $P_3(T)$ which vanishes under $\{d_i\}_{i=1:6}$ and resolves $\{d_i\}_{i=7:10}$. 

Evidently, 
$$
P_T={\rm span}\{\lambda_i^2,\lambda_i\lambda_{i+1},\lambda_i^2\lambda_{i+1}-\lambda_i\lambda_{i+1}^2,\Lambda\}_{i=1:3}.
$$
Direct calculation leads to 
\begin{equation*}
\fint_{T}(\lambda_k^2\lambda_{k+1}-\lambda_k\lambda_{k+1}^2)(\lambda_i^2\lambda_{i+1}-\lambda_i\lambda_{i+1}^2)=\left\{
\begin{array}{rl}
\frac{6}{7!},&k\mtequiv i
\\
-\frac{2}{7!},&k\mtequiv i+1
\\
-\frac{2}{7!}, &k\mtequiv i+2
\end{array}
\right.;
\end{equation*}
further, 
$$
d_k(\lambda_i^2\lambda_{i+1}-\lambda_i\lambda_{i+1}^2)=0,\ \ \mbox{if}\ k\not\in\{7,8,9\}, \ \mbox{and}\ 
d_k(v)=1/60\ \delta_{10,k}.
$$
Summing all above leads to the result by an algebraic calculation. 
\end{proof}
Associated with the finite element {\bf FE${}_{\rm nsc}$}, define
$$
A^3_{h}:=\{w_h\in L^2(\Omega):w_h|_T\in P_3(T); w_h(a)\ \mbox{is\ continuous\ at}\ a\in\mathcal{X}_h;\ \fint_e \llt w_h\rrt=0\ \forall\, e\in\mathcal{E}_h^i\},
$$
and, associated with the boundary condition of $H^1_0(\Omega)$, define
$$
\quad A^3_{h0}:=\{w_h\in A^3_h:w_h(a)=0\ at\ a\in\mathcal{X}_h^b; \fint_ew_h=0,\ \ e\in \mathcal{E}_h^b\}.
$$

\subsubsection{A new scalar quartic element {\bf FE${}_{\rm nsq}$}}

The finite element {\bf FE${}_{\rm nsq}$} is defined by $(T, P_T, D_T)$ with

\fbox{
\begin{minipage}{0.95\textwidth}

\begin{itemize}
\item $T$ is a triangle;
\item $P_T=P_4(T)$;
\item the components of $D_T$ for $v\in H^2(T)$ are:
\begin{multline*}
\qquad
d_i(v)=v(a_i);\ d_{i+3}(v)=\fint_{e_i}v;\ d_{i+6}(v)=\fint_{e_i}\lambda_{i+1}v;
\\ 
d_{i+9}(v)=\fint_{e_i}\partial_{\mathbf n}v;\ d_{i+12}(v)=\fint_T\lambda_{i}v; \ i=1,2,3.\qquad\qquad
\end{multline*}
\end{itemize}
\end{minipage}
}
\begin{lemma}
The finite element {\bf FE${}_{\rm nsq}$} is well defined. 
\end{lemma}
\begin{proof}
There exist $\xi_i\in P_3(T)$, $i=1:9$, such that $d_i(\xi_j)=\delta_{ij}$, $i,j=1:9$. To show the unisolvence of the space, it suffices to show there exist a subset of $P_4(T)$ which vanishes under $\{d_i\}_{i=1:9}$ and resolves $\{d_i\}_{i=10:15}$. 
Set, for $i=1:3$
$$
\varphi_i=\lambda_i^3\lambda_{i+1}-3\lambda_i^2\lambda_{i+1}^2+\lambda_i\lambda_{i+1}^3,\quad\mbox{and}\quad \psi_i=\lambda_i\Lambda.
$$
Then $d_i(\varphi_j)=d_i(\psi_j)=0$, $i=1:9$, $j=1:3$.  Further, 
\begin{equation*}
\fint_{e_k}\partial_{\mathbf{n}}\varphi_i=\left\{
\begin{array}{rl}
-\frac{1}{4}\|\nabla\lambda_k\|,& k=i
\\
-\frac{1}{4}\|\nabla\lambda_k\|,&k\mtequiv i+1
\\
0,&k\mtequiv i+2
\end{array}
\right.,
\quad
\fint_T\lambda_k\varphi_i=0,\ k=1:3,
\end{equation*}

$$
\fint_{e_k}\partial_{\mathbf{n}}\psi_i=\left\{
\begin{array}{rl}
0,& k=i
\\
-\frac{1}{12}\|\nabla\lambda_k\|,&k\mtequiv i+1
\\
-\frac{1}{12}\|\nabla\lambda_k\|,&k\mtequiv i+2
\end{array}
\right.,
\quad\mbox{and}\quad
\fint_T\lambda_k\varphi_i=\left\{
\begin{array}{rl}
\frac{12}{7!},&k\mtequiv i
\\
\frac{8}{7!},&k\mtequiv i+1
\\
\frac{8}{7!},&k\mtequiv i+2.
\end{array}
\right.
$$
A simple algebraic calculation shows that $\{\varphi_i,\psi_i\}_{i=1:3}$, while vanishes under $\{d_i\}_{i=1:9}$, resolves $\{d_i\}_{i=10:15}$. The proof is completed. 
\end{proof}

Corresponding to the finite element {\bf FE${}_{\rm nsq}$}, define space 
\begin{multline*}
A^4_{h}:=\{w_h\in L^2(\Omega):w_h|_T\in P_4(T); w_h(a)\ \mbox{is\ continuous\ at}\ a\in\mathcal{X}_h;
\\ 
\fint_e p_e \llt w_h\rrt=0,\ \forall\,p_e\in P_1(e), \ \mbox{and}\ \fint_e \llt\partial_{\bf n}w_h\rrt=0\ \forall\,e\in\mathcal{E}_h^i\},
\end{multline*}
and, corresponding to the boundary condition of $H^2_0(\Omega)$, define 
$$
A^4_{h0}:=\{w_h\in A^4_h:w_h(a)=0\ at\ a\in\mathcal{X}_h^b; \fint_ew_hp_e=0\ \forall\,p_e\in P_1(e)\ \mbox{and}\ \fint_e\partial_{\bf n}w_h=0,\ \ e\in\mathcal{E}_h^b\}.
$$

\subsection{Enriched cubic and quartic elements}

\subsubsection{An enriched cubic element}
The finite element {\bf FE${}_{\rm ec}$} is defined by $(T, P_T, D_T)$ with

\fbox{
\begin{minipage}{0.90\textwidth}

\begin{itemize}
\item $T$ is a triangle;
\item $P_T=P^{3+}_T:=P_3(T)+\Lambda P_1(T)$;
\item the components of $D_T$ for $v\in H^2(T)$ are:
$$
d_i(v)=v(a_i);\ d_{i+3}(v)=\fint_{e_i}v;\ d_{i+6}(v)=\fint_{e_i}\partial_{\mathbf n}v;\ d_{i+9}(v)=\fint_{e_i}\lambda_{i+1}\partial_{\mathbf n}v;\ i=1,2,3.
$$
\end{itemize}
\end{minipage}
}

\begin{lemma}
The finite element {\bf FE${}_{\rm ec}$} is well-defined. 
\end{lemma}
\begin{proof}
To show  {\bf FE${}_{\rm ec}$} is well-defined, it suffices to show $P^{3+}_T$ is unisolvent versus $D_T$.

Note that $\Lambda\in P_T^{3+}$ and $\dim(P_3(T))=12$. Actually, 
$$
P_T^{3+}=P_2(T)\oplus{\rm span}\{\lambda_i^2\lambda_{i+1}-\lambda_i\lambda_{i+1}^2,\ \lambda_i\Lambda\}_{i=1:3}:=P_2(T)\oplus\hat{P}^{3+}(T).
$$
Set $\eta_i=3\lambda_i^2-2\lambda_i$ and $\eta_{i+3}=\lambda_i\lambda_{i+1}$, $i=1:3$, then $P_2(T)={\rm span}\{\eta_i\}_{i=1:6}$, and $d_i(\eta_j)=\delta_{ij}$. Meanwhile, $d_i(\psi)=0$, $i=1:6$, for any $\psi\in \hat{P}^{3+}(T)$. This way, to show the unisolvence of the finite element, it suffices to show $\{d_i\}_{i=7:12}$ is unisolvent versus $\hat{P}^{3+}(T)$.

We rewrite $f_i(v)=\fint_{e_i}\partial_{\mathbf n}v$, and $g_i(v)=\fint_{e_i}(1/2-\lambda_{i+1})\partial_{\mathbf n}v$, $i=1:3$; then $\{f_i,g_i\}_{i=1:3}=\{d_i\}_{i=7:12}$. Direct calculation leads to that 
\begin{equation*}
f_k(\lambda_i^2\lambda_{i+1}-\lambda_i\lambda_{i+1}^2)=\left\{
\begin{array}{rl}
\frac{1}{3}\|\nabla \lambda_k\|,&k\mtequiv i
\\
-\frac{1}{3}\|\nabla \lambda_k\|,&k\mtequiv i+1
\\
0, &k\mtequiv i+2
\end{array}
\right.,
\quad
f_k(\lambda_i\Lambda)=\left\{
\begin{array}{rl}
0,&k\mtequiv i
\\
-\frac{1}{12}\|\nabla\lambda_k\|&k\mtequiv i+1
\\
-\frac{1}{12}\|\nabla\lambda_k\|&k\mtequiv i+2
\end{array}
\right.,
\end{equation*}

\begin{equation*}
g_k(\lambda_i^2\lambda_j-\lambda_i\lambda_j^2)=-\frac{1}{12}\|\nabla\lambda_k\|,\quad\mbox{and}\quad
g_k(\lambda_i\Lambda)=\left\{
\begin{array}{rl}
0,&k\mtequiv i
\\
-\frac{1}{120}\|\nabla\lambda_k\|,&k\mtequiv i+1
\\
\frac{1}{120}\|\nabla\lambda_k\|,&k\mtequiv i+2.
\end{array}
\right.
\end{equation*}
A further algebraic calculation (of a determinant) proves the unisolvence of $\{d_i\}_{i=7}^{12}$ versus $\hat{P}^3(T)$. More clearly, for $i,j,k,l\in\{1,2,3\}$,
$$
f_i(\phi_j)=\delta_{ij},\ \  g_k(\phi_l)=0,\ \  f_i(\psi_j)=0, \ \ \mbox{and}\ \ g_k(\psi_l)=\delta_{kl},
$$
where
$$
\phi_i=\frac{-1}{\|\nabla\lambda_i\|}\lambda_i(2\lambda_i-1)(\lambda_i-1)
$$
and
$$
\psi_i=\frac{1}{\|\nabla\lambda_i\|}\left[2(\lambda_i^2\lambda_{i+1}-\lambda_i\lambda_{i+1}^2)-8(\lambda_{i+1}^2\lambda_{i+2}-\lambda_{i+1}\lambda_{i+2}^2)+2(\lambda_{i+2}^2\lambda_i-\lambda_{i+2}\lambda_i^2)+40(\lambda_{i+1}-\lambda_{i+2})\Lambda\right]
$$
form a basis of $\hat{P}^3(T)$. The lemma is proved. 
\end{proof}

Associated with the finite element {\bf FE${}_{\rm ec}$}, define
\begin{multline*}
B_h^{3+}:=\{v\in L^2(\Omega)|\ v|_T\in P_3^+(T), v\ \mbox{is\ continuous\ at}\ a\in\mathcal{X}_h;
\\ 
\int_e\llt v\rrt=0, \ \mbox{and}\  \int_ep_e\llt \partial_{\bf n}v\rrt=0,  \forall\, p_e\in P_1(e), \ \forall\,e\in\mathcal{E}_h^i\},
\end{multline*}
and associated with the boundary condition of $H^2_0(\Omega)$, define
$$
B_{h0}^{3+}:=\{v\in B_h^{3+}: v(a)=0,\ a\in\mathcal{X}_h^b;\ \int_ev=0,\ e\in\mathcal{E}_h^b; \int_ep_e\partial_{\bf n}v=0,\  \forall\,p_e\in P_1(e),\ e\in\mathcal{E}_h^b\}.
$$

\subsubsection{Enriched quartic element {\bf FE${}_{\rm eq}$}}

The element {\bf FE${}_{\rm eq}$} is defined by $(T, P_T, D_T)$ with

\fbox{
\begin{minipage}{0.95\textwidth}
\begin{itemize}
\item $T$ is a triangle;
\item $P_T=P^{4+}_T:=P_4(T)+{\rm span}\{(\lambda_i^2\lambda_{i+1}-\lambda_i\lambda_{i+1}^2)\Lambda\}$;
\item the components of $D_T$ for $v\in H^2(T)$ are:
\begin{multline*}
d_i(v)=v(a_i);\ d_{i+3}(v)=\fint_{e_i}v;\ d_{i+6}(v)=\fint_{e_i}\lambda_{i+1}v;
\\ 
d_{i+9}(v)=\fint_{e_i}\partial_{\mathbf n}v;\ d_{i+12}(v)=\fint_{e_i}\lambda_{i+1}\partial_{\mathbf n}v;\ d_{i+15}(v)=\fint_{e_i}\lambda_{i+1}^2\partial_{\bf n}v; \ i=1,2,3.
\end{multline*}
\end{itemize}
\end{minipage}
}

\begin{lemma}
The finite element {\bf FE${}_{\rm eq}$} is well-defined. 
\end{lemma}

\begin{proof}
There exist $\xi_i\in P_3(T)$, $i=1:9$, such that $d_i(\xi_j)=\delta_{ij}$, $i,j=1:9$. To show the unisolvence of the space, it suffices to show there exist a subset of $P^{4+}_T$ which vanishes under $\{d_i\}_{i=1:9}$ and resolves $\{d_i\}_{i=10:18}$. Set for $i=1:3$, 
$$
\varphi_i=\lambda_i^3\lambda_{i+1}-3\lambda_i^2\lambda_{i+1}^2+\lambda_i\lambda_{i+1}^3,\quad
\psi_i=\lambda_i\Lambda,\quad
\eta_i=(\lambda_i^2\lambda_{i+1}-\lambda_i\lambda_{i+1}^2)\Lambda,
$$
and
$$
f_{i}=d_{i+9}, \ \ g_{i}=d_{i+12},\ \ \mbox{and}\ \ h_{i}=d_{i+15},\ i=1:6. 
$$
Direct calculation leads to that $d_i(\varphi_j)=d_i( \psi_j)=d_i( \eta_j)=0$, $i=1:9$ and $j=1:3$. Further,

\begin{equation*}
f_k(\varphi_i)=\left\{
\begin{array}{rl}
-\frac{1}{4}\|\nabla\lambda_k\|,& k=i
\\
-\frac{1}{4}\|\nabla\lambda_k\|,&k\mtequiv i+1
\\
0,&k\mtequiv i+2
\end{array}
\right.
;\
g_k(\varphi_i)=\left\{
\begin{array}{rl}
-\frac{1}{5}\|\nabla\lambda_k\|,&k\mtequiv i
\\
-\frac{1}{20}\|\nabla\lambda_k\|,&k\mtequiv i+1
\\
0,&k\mtequiv i+2
\end{array}
\right.
;\
h_k(\varphi_i)=\left\{
\begin{array}{rl}
-\frac{1}{6}\|\nabla\lambda_k\|,&k\mtequiv i
\\
-\frac{1}{60}\|\nabla\lambda_k\|,&k\mtequiv i+1
\\
\frac{1}{60}\|\nabla\lambda_k\|,&k\mtequiv i+2
\end{array}
\right.;
\end{equation*}

$$
f_k(\psi_i)=\left\{
\begin{array}{rl}
0,& k=i
\\
-\frac{1}{12}\|\nabla\lambda_k\|,&k\mtequiv i+1
\\
-\frac{1}{12}\|\nabla\lambda_k\|,&k\mtequiv i+2
\end{array}
\right.
;\
g_k(\psi_i)=\left\{
\begin{array}{rl}
0,&k\mtequiv i
\\
-\frac{1}{30}\|\nabla\lambda_k\|,&k\mtequiv i+1
\\
-\frac{1}{20}\|\nabla\lambda_k\|,&k\mtequiv i+2
\end{array}
\right.
;\
h_k(\psi_i)=\left\{
\begin{array}{rl}
0,&k\mtequiv i
\\
-\frac{1}{60}\|\nabla\lambda_k\|,&k\mtequiv i+1
\\
-\frac{1}{30}\|\nabla\lambda_k\|,&k\mtequiv i+2
\end{array}
\right.;
$$

$$
f_k(\eta_i)=\left\{
\begin{array}{rl}
0,& k=i
\\
0,&k\mtequiv i+1
\\
0,&k\mtequiv i+2
\end{array}
\right.
;\
g_k(\eta_i)=\left\{
\begin{array}{rl}
0,&k\mtequiv i
\\
0,&k\mtequiv i+1
\\
-\frac{1}{420}\|\nabla\lambda_k\|,&k\mtequiv i+2
\end{array}
\right.
;\
h_k(\eta_i)=\left\{
\begin{array}{rl}
0,&k\mtequiv i
\\
0,&k\mtequiv i+1
\\
-\frac{1}{420}\|\nabla\lambda_k\|,&k\mtequiv i+2.
\end{array}
\right.
$$
Therefore, a further algebraic calculation shows that $\{d_i\}_{i=10:18}$ is resolved by ${\rm span}\{\varphi_i,\psi_i,\eta_i\}_{i=1:3}$. Actually, set
$$
\left\{
\begin{array}{l}
\displaystyle b_i^0=\frac{1}{\|\nabla\lambda_i\|}\left[-6(\varphi_i+\varphi_{i+2})+6(\psi_i+2\psi_{i+1}+2\psi_{i+2})+210(\eta_i+\eta_{i+1}+\eta_{i+2})\right]
\\
\displaystyle b_i^1=\frac{1}{\|\nabla\lambda_i\|}\left[ 30(\varphi_i+\varphi_{i+2})-90(\psi_{i+1}+\psi_{i+2})+420(-3\eta_i-\eta_{i+1}+3\eta_{i+2})\right]
\\
\displaystyle b_i^2=\frac{1}{\|\nabla\lambda_i\|}\left[-30(\varphi_i+\varphi_{i+2})+90(\psi_{i+1}+\psi_{i+2})+1260(\eta_i-\eta_{i+2})\right]
\end{array}
\right.,
$$
then $d_{i+9+3k}(b^l_j)=\delta_{ij}\delta_{kl}$. The proof is completed. 
\end{proof}

Associated with the finite element {\bf FE${}_{\rm eq}$}, define
\begin{multline*}
B_h^{4+}:=\{v\in L^2(\Omega):v|_T\in P_4^+(T), v\ \mbox{is\ continuous\ at}\ a\in\mathcal{X}_h;
\\ 
\int_ep_e\llt v\rrt=0,\ \forall\, p_e\in P_1(e),\  \mbox{and}\ \int_eq_e\llt\partial_{\bf n}v\rrt=0,\ \forall\, q_e\in P_2(e), \ \forall\, e\in\mathcal{E}_h^i\}.
\end{multline*}
and, associated with the boundary condition of $H^2_0(\Omega)$, define, 
$$
B_{h0}^{4+}:=\{v\in B_h^{4+}: v(a)=0,\ a\in\mathcal{X}_h^b;\ \int_ep_ev=0,\ \forall\, p_e\in P_1(e), \ \mbox{and}\  \int_eq_e\partial_{\bf n}v=0,\ \forall\,q_e \in P_2(e),\ \forall\,e\in\mathcal{E}_h^b\}.
$$

\subsection{Enriched vector quadratic and cubic elements}

\subsubsection{A vector enriched quadratic element {\bf FE${}_{\rm veq}$}}

The finite element {\bf FE${}_{\rm veq}$} is defined by $(T, P_T, D_T)$ with

\fbox{
\begin{minipage}{0.95\textwidth}

\begin{itemize}
\item $T$ is a triangle;
\item $P_T=\uP{}_T^{2+}:=(P_2(T))^2+{\rm span}\{\nabla(\lambda_1\Lambda),\nabla(\lambda_2\Lambda)\}$;
\item for $v\in H^1(T)$ set $d_i(v)=\fint_{e_i}v,\ f_i(v)=\fint_{e_i}\lambda_{i+1}v,\  i=1,2,3,$ and $g(v)=\fint_T v$, and the components of $D_T$ are: 
\begin{multline*}
\qquad D_i(\uv)=d_i(\uv^1),\ D_{i+3}(\uv)=f_i(\uv^1),\ D_7=g(\uv^1),
\\
D_{i+7}(\uv)=d_i(\uv^2),\ D_{i+10}(\uv)=f_i(\uv^2), \ D_{14}(\uv)=g(\uv^2),\ \ i=1:3.\qquad
\end{multline*}

\end{itemize}
\end{minipage}
}
Note that $\sum_{i=1}^3\lambda\Lambda=\Lambda\in P_3(T)$, and thus $\uP{}^{2+}_T=\uP{}_2(T)+{\rm span}\{\nabla(\lambda_i\Lambda)\}_{i=1:3}$.

\begin{lemma}
The finite element {\bf FE${}_{\rm veq}$} is well-defined. 
\end{lemma}
\begin{proof}
To show the finite element is well-defined, it suffices to show $P_T$ is unisolvent versus $D_T$. 

The shape function space can be rewritten as 
$$
\uP{}_T^{2+}=({\rm span}\{\lambda_i^2,\lambda_i\lambda_{i+1}\}_{i=1}^3)^2\oplus {\rm span}\{\nabla(\lambda_1\Lambda),\nabla(\lambda_2\Lambda)\}.
$$

Direct calculation leads to that, $i,k=1:3$, $l=1,2$, 
$$
d_k(\lambda_i^2)=\left\{
\begin{array}{rl}
0,&i=k \\ \frac{1}{3},&k\neq i
\end{array}
\right.;
\quad
f_k(\lambda_i^2)=\left\{
\begin{array}{rl}
0, &k\mtequiv i
\\
\frac{1}{12},&k\mtequiv i+1
\\
\frac{1}{4},&k\mtequiv i+2
\end{array}
\right.;
\quad
g(\lambda_i^2)=\frac{1}{6};
$$
$$
d_k(\lambda_i\lambda_{i+1})=\left\{
\begin{array}{rl}
0,& k=i
\\
0,&k\mtequiv i+1
\\
\frac{1}{6},& k=i+2
\end{array}
\right.;
\quad
f_k(\lambda_i\lambda_{i+1})=\left\{
\begin{array}{rl}
0,&k\mtequiv i,
\\
0,&k\mtequiv i+1
\\
\frac{1}{12},&k\mtequiv i+2
\end{array}
\right.;
\quad
g(\lambda_i\lambda_{i+1})=\frac{1}{12};
$$
$$
d_k(\partial_l(\lambda_i\Lambda))=\left\{
\begin{array}{rl}
0,& k=i
\\
\frac{1}{12}\partial_l\lambda_k,&k\neq i
\end{array}
\right.;
\quad
f_k(\partial_l(\lambda_i\Lambda))=\left\{
\begin{array}{rl}
0,&k\mtequiv i
\\
\frac{1}{90}\partial_l\lambda_k,&k\mtequiv i+1
\\
\frac{1}{60}\partial_l\lambda_k,&k\mtequiv i+2
\end{array}
\right.;
\quad
g(\partial_l(\lambda_i\Lambda))=0.
$$

Now, set $\eta_i=\lambda_i^2$, $\eta_i=\lambda_i\lambda_{i+1}$, and rewrite the basis functions of $\uP{}^{2+}_T$ as:
$$
\phi_i=(\eta_i,0)^\top,\ \phi_{i+6}=(0,\eta_i)^\top,\  i=1:6;\ \phi_{13}=\nabla(\lambda_1\Lambda),\ \phi_{14}=\nabla(\lambda_2\Lambda).
$$
Let $M$ be the stiffness matrix with $M_{ij}=D_i(\phi_j)$. Then direct calculation leads to that 
$$
\det(M)=\frac{103}{501530650214400}(\nabla\lambda_1\cdot\curl\lambda_2).
$$
Note that $\nabla\lambda_1\cdot\curl\lambda_2=\nabla\lambda_1\cdot(\nabla\lambda_2)^\perp\neq 0$, (otherwise, $\nabla\lambda_1\parallel \nabla\lambda_2$, namely $e_1\parallel e_2$), the stiffness matrix $M$ is non-singular, and the unisolvence is proved. 
\end{proof}

Associated with the finite element {\bf FE${}_{\rm veq}$}, define

$$
\uG{}_h^{2+}:=\{\uv\in \uL^2(\Omega):\uv|_T\in \uP{}_2^+(T),\ \int_ep_e\llt\uv^j\rrt=0, \ \forall\, p_e\in P_1(e),\ \forall\, e\in\mathcal{E}_h^i,\ j=1,2\},
$$
and, corresponding to the boundary condition of $\uH{}^1_0(\Omega)$, define
$$
\uG{}_{h0}^{2+}:=\{\uv\in \uG{}_h^{2+}: \int_ep_e\uv^j=0,\ \forall\,p_e\in P_1(e),\ \forall\, e\in\mathcal{E}_h^b, j=1,2\}.
$$

\subsubsection{A vector enriched cubic element {\bf FE${}_{\rm vec}$}}

The finite element {\bf FE${}_{\rm vec}$} is defined by $(T, P_T, D_T)$ with

\fbox{
\begin{minipage}{0.95\textwidth}

\begin{itemize}
\item $T$ is a triangle;
\item $P_T=\uP{}_T^{3+}:=(P_3(T))^2+{\rm span}\{\nabla(\lambda_i^2\lambda_{i+1}-\lambda_i\lambda_{i+1}^2)\Lambda\}_{i=1:3}$;
\item the components of $D_T$ for $v\in H^2(T)$ are:
\begin{multline*}
d_i(\uv)=\fint_{e_i}\uv^1,\ d_{i+3}(\uv)=\fint_{e_i}\lambda_{i+1}\uv^1,\ d_{i+6}(\uv)=\fint_{e_i}\lambda_{i+1}^2\uv^1,\ i=1,2,3,\ d_{10}(\uv)=\fint_T \uv^1;
\\
d_i(\uv)=\fint_{e_i}\uv^2,\ d_{i+3}(\uv)=\fint_{e_i}\lambda_{i+1}\uv^2,\ d_{i+6}(\uv)=\fint_{e_i}\lambda_{i+1}^2\uv^2,\ i=11,12,13,\ d_{20}(\uv)=\fint_T \uv^2;
\\
d_{j+20}(\uv)=\fint_T \upsi{}_j \uv,\ j=1:3,\qquad\qquad\qquad\qquad\qquad\qquad\qquad\qquad
\\ \mbox{where}\ \upsi{}_j\in \uP{}_3(T)+{\rm span}\{\nabla(\lambda_i^2\lambda_{i+1}-\lambda_i\lambda_{i+1}^2)\Lambda\},\ \mbox{and}\ d_i(\upsi{}_j)=0,\ i=1:20.
\end{multline*}
\end{itemize}
\end{minipage}
}

\begin{lemma}
The finite element {\bf FE${}_{\rm vec}$} is well-defined. 
\end{lemma}
\begin{proof}
There exist $\{\xi_j\}_{j=1:10}\subset P_3(T)$, such that $d_i((\xi_j,0)^\top)=\delta_{ij}$, $1\leqslant i,j\leqslant 10$. Therefore, the existence of $\upsi{}_j$, $j=1:3$, is confirmed. Moreover, $P_T=\uP{}_3(T)+{\rm span}\{\upsi{}_j\}_{j=1:3}$. The unisolvence of $D_T$ versus $P_T$ can be proved by standard approach. The proof is completed. 
\end{proof}

Associated with the finite element {\bf FE${}_{\rm vec}$}, define

$$
\uG{}_h^{3+}:=\{\uv\in \uL^2(\Omega):\uv|_T\in \uP{}_3^+(T),\ \int_ep_e\llt\uv^j\rrt=0,\ \forall\, p_e\in P_2(e),\ \forall\, e\in\mathcal{E}_h^i,\ j=1,2\},
$$
and, corresponding to the boundary condition of $\uH{}^1_0(\Omega)$, define
$$
\uG{}_{h0}^{3+}:=\{\uv\in \uG{}_h^{3+}: \int_ep_e\uv^j=0\ \forall\, p_e\in P_2(e),\ e\in\mathcal{E}_h^b,\ j=1,2\}.
$$

\section{An optimal piecewise cubic finite element method for biharmonic equation}
\label{sec:cubic}

In this section, we consider the biharmonic equation: 
\begin{equation}
\left\{
\begin{array}{rl}
\displaystyle\Delta^2u=f&\mbox{in}\ \Omega;
\\
\displaystyle u=0,&\mbox{on}\ \partial\Omega.
\\
\displaystyle\partial_{\mathbf n}u=0,&\mbox{on}\ \partial\Omega.
\end{array}
\right.
\end{equation}
The variational problem is to find $u\in H^2_0(\Omega)$, such that 
\begin{equation}\label{eq:bhvp}
(\nabla^2u,\nabla^2v)=(f,v),\quad\forall\,v\in H^2_0(\Omega).
\end{equation}
Now given a triangulation $\mathcal{T}_h$, define
\begin{multline*}
B_h^{3}:=\{v\in L^2(\Omega)|\ v|_T\in P_3(T), v\ \mbox{is\ continuous\ at}\ a\in\mathcal{X}_h;
\\ 
\int_e\llt v\rrt=0,\ \mbox{and}\  \int_ep_e\llt\partial_{\bf n}v\rrt=0, \forall\, p_e\in P_1(e),\ \forall\, e\in\mathcal{E}_h^i\},
\end{multline*}
and
$$
B_{h0}^{3}:=\{v\in B_h^{3}: v(a)=0,\ a\in\mathcal{X}_h^b;\ \int_ev=0,\ \mbox{and}\ \int_ep_e\partial_{\bf n}v=0,\  \forall\, p_e\in P_1(e),\ \forall\,e\in\mathcal{E}_h^b\}.
$$
\begin{remark}\label{rem:strB3dg}
Evidently, $B^3_{h0}=\{w_h\in B^{3+}_{h0}:w_h|_T\in\mathcal{T}_h,\ \forall\,T\in\mathcal{T}_h\}$ and $B_{h0}^{3}=\{v_h\in A^3_{h0}: \ \int_ep_e\llt\partial_{\bf n}v\rrt=0,\ \forall\, p_e\in P_1(e), \ \forall\,e\in\mathcal{E}_h\}.$
\end{remark}
The discretization is to find $u_h\in B^3_{h0}$, such that 
\begin{equation}\label{eq:bhvpp3}
(\nabla_h^2u_h,\nabla^2_hv_h)=(f,v_h),\quad\forall\,v_h\in B^3_{h0}.
\end{equation}
By the weak continuity of $B^3_{h0}$, the discrete Poincar\'e-Friedrichs inequality holds as below. The well-posedness of \eqref{eq:bhvpp3} is confirmed. 
\begin{lemma}
(c.f., e.g, \cite{Brenner.S;Wang.K;Zhao.J2004}) There exists a constant $C$ depending on the regularity of the grids only such that 
\begin{equation}
\|w_h\|_{0,\Omega}^2+\|\nabla_hw_h\|_{0,\Omega}^2\leqslant C\|\nabla_h^2w_h\|_{0,\Omega}^2,\ \ \forall\,w_h\in B^3_{h0}+H^2_0(\Omega).
\end{equation}
\end{lemma}

The main result of this section is the theorem below.
\begin{theorem}\label{thm:ratep3}
Let $u$ and $u_h$ be the solutions of \eqref{eq:bhvp} and \eqref{eq:bhvpp3}, respectively. If $u\in H^4(\Omega)$, then, with $C$ a generic constant depending on $\Omega$ and the regularity of the grid only,
\begin{equation}\label{eq:h2error3}
\|\nabla_h^2(u-u_h)\|_{0,\Omega}\leqslant Ch^2|u|_{4,\Omega}.
\end{equation}
When $\Omega$ is convex, 
\begin{equation}\label{eq:h1error3}
\|\nabla_h(u-u_h)\|_{0,\Omega}\leqslant Ch^3|u|_{4,\Omega}.
\end{equation}
\end{theorem}

We postpone the proof of Theorem \ref{thm:ratep3} after some technical lemmas.

\subsection{Structural properties of $B^3_{h0}$}
\subsubsection{An auxiliary discretized Stokes complexes}

\begin{lemma}\label{lem:dsc3+}
A discretized Stokes complex holds as below:
\begin{equation}
\begin{array}{ccccccccc}
0 & \longrightarrow & B_{h0}^{3+} & \xrightarrow{\bs{\mrm{\nabla}}_h} & \uG{}_{h0}^{2+} & \xrightarrow{\mrm{rot}_h} & \mathbb{P}_{h0}^1  & \longrightarrow & 0.
\end{array}
\end{equation}
\end{lemma}
\begin{proof}
By the construction of $\uG{}^{2+}_{h0}$ and Lemma \ref{lem:pGL} which shows $\rot_h\uG{}^{2}_{h0}=\mathbb{P}^1_{h0}$, $\rot_h\uG{}^{2+}_{h0}=\mathbb{P}^1_{h0}$. Note $\dim(\nabla_hB_{h0}^{3+})=\dim(B_{h0}^{3+})=\#(\mathcal{X}^i_h)+3\#(\mathcal{E}_h^i)$, $\dim(\uG{}_{h0}^{2+})=2(\#(\mathcal{T}_h)+2\#(\mathcal{E}_h^i))$ and $\dim(\mathbb{P}_{h0}^1)=3\#(\mathcal{T}_h)-1$, and thus by Euler formula, $\dim(\nabla_hB_{h0}^{3+})+\dim(\mathbb{P}_{h0}^1)=\dim(\uG{}_{h0}^{2+})$, namely $\dim(\nabla_h B^{3+}_{h0})=\dim (\{\upsi{}_h\in\uG{}_{h0}^{2+}:\rot_h\upsi{}_h=0\})$. Further, evidently, $\nabla_hB_{h0}^{3+}\subset \{\upsi{}_h\in\uG{}_{h0}^{2+}:\rot_h\upsi{}_h=0\}$. Therefore, $\nabla_hB_{h0}^{3+}= \{\upsi{}_h\in\uG{}_{h0}^{2+}:\rot_h\upsi{}_h=0\}$, and the exact complex follows. 
\end{proof}

\subsubsection{On the structure of $B^3_{h0}$: vertical perspectives}

\begin{theorem}\label{thm:exactb3}
A discretized Stokes complex holds as below:
\begin{equation}
\begin{array}{ccccccccc}
0 & \longrightarrow & B_{h0}^{3} & \xrightarrow{\bs{\mrm{\nabla}}_h} & \uG{}_{h0}^{2} & \xrightarrow{\mrm{rot}_h} & \mathbb{P}_{h0}^1  & \longrightarrow & 0.
\end{array}
\end{equation}
\end{theorem}
\begin{proof}
By Lemma \ref{lem:pGL}, it suffices to prove $\nabla_hB_{h0}^3=\uG{}^2_{h0}(\rot_h,0):=\{\uv{}_h\in\uG{}^2_{h0}:\rot_h\uv{}_h=0\}$. It is evident that $\nabla_hB_{h0}^3\subset\uG{}^2_{h0}(\rot_h,0)$. On the other hand, given $\uv{}_h\in \uG{}^2_{h0}\subset \uG{}^{2+}_{h0}$ such that $\rot_h\uv{}_h=0$, by Lemma \ref{lem:dsc3+}, there exists $w_h\in B^{3+}_{h0}$, such that $\nabla_hw_h=\uv{}_h$. Since $\nabla(w_h|_T)=\uv{}_h|_T\in \uP^2(T)$, $w_h|_T\in P_3(T)$, and thus $w_h\in B^3_{h0}$. Namely $\nabla_hB_{h0}^3\supset \uG{}^2_{h0}(\rot_h,0)$. The proof is completed. 
\end{proof}

\begin{lemma}\label{eq:imp3}
$B^3_{h0}=\{w_h\in A^3_{h0}:\nabla w_h\in \uG{}^2_{h0}\}$. 
\end{lemma}
\begin{proof}
By the constructions of $B^3_{h0}$ and $A^3_{h0}$ and Theorem \ref{thm:exactb3}, $B^3_{h0}\subset\{w_h\in A^3_{h0}:\nabla w_h\in \uG{}^2_{h0}\}$. On the other hand, given $w_h\in A^3_{h0}$ such that $\nabla_hw_h\in \uG{}^2_{h0}$, there exists $w_h'\in B^3_{h0}$, such that $\nabla_hw_h'=\nabla_hw_h$. Since $\nabla_h$ is injective on $A^3_{h0}$, $w_h=w_h'$. The proof is completed. 
\end{proof}

\subsubsection{On the structure of $B^3_{h0}$: a horizontal perspective} Now we are going to show that $B^3_{h0}$ admits a set of basis functions with vertex-patch-based supports. Define 
\begin{equation}
\uS{}^2_{h0}(\rot,w0):=\{\uv\in\uS{}^2_{h0}:(\rot\uv,q_h)=0,\ \forall\,q_h\in \mathbb{P}^0_{h0}\}.
\end{equation} 
As the $\uS{}^2_{h0}-\mathbb{P}^0_{h0}$ pair is stable for Stokes problem, 
$$
\dim(\uS{}^2_{h0}(\rot,w0))=\dim(\uS{}^2_{h0})-\dim(\mathbb{P}^0_{h0})=\#(\mathcal(N)_{h0}^i)+3\#(\mathcal{X}_{h0}^i)=\dim(\uG{}^2_{h0}(\rot,0)).
$$

For $a\in\mathcal{X}_h$, denote by $P_a$ the union of triangles of which $a$ is a vertex, namely the patch associated with $a$; for $e\in\mathcal{E}_h$,  denote by $P_e$ the patch associated with $e$. Define, with respect to $a\in\mathcal{X}_h$ and $e\in\mathcal{E}_h$,
\begin{description}
\item[$\uphi{}_a^x$] $\uphi{}_a^x\in\uS{}^2_h$, $\uphi{}_a^x(a)=(1,0)^\top$; $\uphi{}_a^x(a')=\undertilde{0}$ on $a\neq a'\in\mathcal{X}_h$; $\fint_{e'}\uphi{}_a^x=\undertilde{0}$ on $e'\in\mathcal{E}_h$;
\item[$\uphi{}_a^y$] $\uphi{}_a^y\in\uS{}^2_h$, $\uphi{}_a^y(a)=(0,1)^\top$; $\uphi{}_a^y(a')=\undertilde{0}$ on $a\neq a'\in\mathcal{X}_h$; $\fint_{e'}\uphi{}^y_a=\undertilde{0}$ on $e'\in\mathcal{E}_h^i$;
\item[$\uphi{}_{P_a}$] $\uphi{}_{P_a}\in\uS{}^2_h$, $\uphi{}_a^x(a')=\undertilde{0}$ on $ a'\in\mathcal{X}_h$; $\fint_e\uphi=\undertilde{0}$ on $e\in\mathcal{E}_h$ and $a\not\in e$; $\fint_e\uphi\cdot\mathbf{t}_{e,P_a}=1$ and $\int_e\uphi\cdot\mathbf{n}_{e,P_a}=0$ on $e\subset P_a$ and $a\in e$, where $\mathbf{t}_{e,P_a}$ is the unit tangential vector along $e$ starting from $a$ and $\mathbf{n}_{e,P_a}$ is the normal direction of $e$ along the anticlockwise with respect to $P_a$;
\item[$\uphi{}_e$] $\uphi{}_e\in\uS{}^2_h$, $\fint_e\uphi{}_e\cdot\utau{}_e=0$, $\fint_e\uphi{}_e\cdot\mathbf{n}_e=1$, and $\uphi{}_e$ vanishes on $\Omega\setminus \mathring{P_e}$.
\end{description}

\begin{lemma}
The set $\{\uphi{}^x_a,\uphi{}^y_a,\uphi{}_{P_a},\uphi{}_e\}_{a\in\mathcal{X}^i_h,\ e\in\mathcal{E}^i_h}$ forms a basis of $\uS{}^2_{h0}(\rot,w0)$; namely
\begin{equation}\label{eq:basisweakrotfree}
\uS{}^2_{h0}(\rot,w0)={\rm span}\{\uphi{}_a^x\}_{a\in\mathcal{X}_h^i}\oplus {\rm span}\{\uphi{}_a^y\}_{a\in\mathcal{X}_h^i}\oplus{\rm span}\{\uphi{}_e\}_{e\in\mathcal{E}_h^i}\oplus{\rm span}\{\uphi{}_{P_a}\}_{a\in\mathcal{X}_h^i}.
\end{equation}
\end{lemma}
\begin{proof}
By direct calculation, the functions $\uphi{}_a^x$, $\uphi{}_a^y$, $\uphi{}_{P_a}$ and $\uphi{}_e$ are all well-defined, and they all belong to $\uS{}^2_{h0}(\rot,w0)$. By their definitions, the functions $\{\uphi{}^x_a,\uphi{}^y_a,\uphi{}_e\}_{a\in\mathcal{X}_h^i,\ e\in\mathcal{E}_h^i}$ are linearly independent, and the summation ${\rm span}\{\uphi{}^x_a,\uphi{}^y_a,\uphi{}_e\}_{a\in\mathcal{X}_h^i,\ e\in\mathcal{E}_h^i}+{\rm span}\{\uphi{}_{P_a}\}_{a\in\mathcal{X}_h^i}$ is direct. It remains for us to show $\{\uphi{}_e\}_{e\in\mathcal{E}_h^i}$ are linearly independent. 

Assume there exist $\{\alpha_a\}_{a\in\mathcal{X}}\subset\mathbb{R}$ with $\alpha_a=0$ for $a\in\mathcal{X}_h^b$, such that  $\upsi=\sum_{a\in \mathcal{X}_{h}}\alpha_a\uphi{}_{P_a}\equiv 0$. By the definition of $\uphi{}_{P_a}$, for any $e\in\mathcal{E}_h^i$, $|\fint_{e}\upsi\cdot\mathbf{n}_e|=|\alpha_{a_e^L}-\alpha_{a_e^R}|$, where $a_e^L$ and $a_e^R$ are the two ends of $e$; thus $\alpha_{a_e^L}=\alpha_{a_e^R}$ for every $e\in\mathcal{E}_h^i$. Since $\alpha_a=0$ for $a\in \mathcal{X}_h^b$, $\alpha_a=0$ for $a\in\mathcal{X}_h^{b,+1}$; recursively, we obtain $\alpha_a=0$ for $a\in\mathcal{X}_h^{b,+j}$ level by level, and finally $\alpha_a=0$ for $a\in\mathcal{X}_h$. 

The proof is completed by noting the two sides of \eqref{eq:basisweakrotfree} have the same dimension. 
\end{proof}

\begin{theorem}
There exists a set of functions $\{w_{h}^i\}_{i=1:\dim(B^3_{h0})}\subset B^3_{h0}$, such that the functions are linearly independent and are each supported in a patch of some vertex. 
\end{theorem}
\begin{proof}
Firstly, define an operator $\mathcal{F}_h:\uS{}^2_{h0}(\rot,w0)\to \uG{}^2_{h0}(\rot_h,0)$ by
$$
\mathcal{F}_h\uphi{}_h=\uphi{}_h+\undertilde{\phi}{}_h,\ \undertilde{\phi}{}_h\in\undertilde{\mathcal{B}}{}^2_{h0},\ \ \ \mbox{such\ that}\ \ \rot_h (\mathcal{F}_h\uphi{}_h)=0.
$$
The summand $\undertilde{\phi}{}_h$ can be determined cell by cell, and ${\rm supp} (\mathcal{F}_h\uphi{}_h)\subset {\rm supp}(\uphi{}_h)$. Secondly, define $(\nabla^{-1})_h:\uG{}^2_{h0}(\rot_h,0)\to B^3_{h0}$ such that $\nabla_h(\nabla^{-1})_h\upsi{}_h=\upsi{}_h$. By the exact relation between the spaces, the operator is well-defined, and ${\rm supp}((\nabla^{-1})_h\upsi{}_h)\subset {\rm supp}(\upsi{}_h)$. It thus follows that $(\nabla^{-1})_h\circ\mathcal{F}_h:\uS{}^2_{h0}(\rot, w0)\to B^3_{h0}$ is bijective and preserves support. Now, set 
$$
\left\{w_h^i\right\}_{i=1}^{3\#(\mathcal{X}^i_h)+\#(\mathcal{E}^i_h)}:=\left[(\nabla^{-1})_h\circ\mathcal{F}_h\right]\left( \{\uphi{}_a^x\}_{a\in\mathcal{X}_h^i}+\{\uphi{}_a^y\}_{a\in\mathcal{X}_h^i}+\{\uphi{}_e\}_{e\in\mathcal{E}_h^i}+\{\uphi{}_{P_a}\}_{a\in\mathcal{X}_h^i}\right),
$$ 
and the set is what we need. The proof is completed. 
\end{proof}

\subsection{Convergence analysis}

\begin{lemma}\label{lem:approxb3}
There exists a constant $C$, such that for any $w\in H^2_0(\Omega)\cap H^4(\Omega)$, there exists $w_h\in B^3_{h0}$ and 
$$
\|\nabla_h^2(w-w_h)\|_{0,\Omega}\leqslant Ch^2|w|_{4,\Omega};
$$
moreover, if $\Omega$ is convex, then $w_h$ can be chosen such that 
$$
\|\nabla_h(w-w_h)\|_{0,\Omega}\leqslant Ch^3|w|_{4,\Omega}.
$$
\end{lemma}
\begin{proof}
Given $w\in H^2_0(\Omega)\cap H^4(\Omega)$, $\uphi:=\nabla w\in\uH{}^1_0(\Omega)\cap \uH^3(\Omega)$, and by the aid of the rotated incompressible Stokes problem, regarding  Lemma \ref{lem:pGL}, there exists $\uphi{}_h\in \uG{}^2_{h0}$, such that $\displaystyle\|\nabla_h(\uphi-\uphi{}_h)\|_{0,\Omega}\leqslant Ch^2|\uphi|_{3,\Omega}$, and, when $\Omega$ is convex, $\displaystyle\|\uphi-\uphi{}_h\|_{0,\Omega}\leqslant Ch^3|\uphi|_{4,\Omega}$. By Theorem \ref{thm:exactb3}, there exists $w_h\in B^3_{h0}$, such that $\nabla_hw_h=\uphi{}_h$ and completes the proof. 
\end{proof}

Define for $\varphi,\psi$ that make it reasonable the bilinear form
\begin{equation}\label{eq:res1def}
\mathcal{R}_h^1(\varphi,\psi):=(\nabla^2\varphi,\nabla_h^2\psi)+(\nabla\Delta \varphi,\nabla_h\psi),
\end{equation}
\begin{equation}\label{eq:res2def}
\mathcal{R}_h^2(\varphi,\psi):=(\nabla\Delta \varphi,\nabla_h\psi)+(\Delta^2\varphi,\psi),
\end{equation}
and 
\begin{equation}\label{eq:resdef}
\mathcal{R}_h(\varphi,\psi):=\mathcal{R}_h^1(\varphi,\psi)-\mathcal{R}_h^2(\varphi,\psi).
\end{equation}
\begin{lemma}\label{lem:res}
There is a constant $C$, such that 
\begin{equation}\label{eq:res1}
\displaystyle\mathcal{R}_h^1(\varphi,w_h)\leqslant Ch^k|\varphi|_{2+k,\Omega}\|\nabla_h^2w_h\|_{0,\Omega},\ \ \forall\,\varphi\in H^2_0(\Omega)\cap H^k(\Omega),\ w_h\in B^{3+}_{h0}+H^2_0(\Omega),\ k=1,2.
\end{equation}
\begin{equation}\label{eq:res2}
\displaystyle\mathcal{R}_h^i(\varphi,w_h)\leqslant Ch^2|\varphi|_{4,\Omega}\|\nabla_h^2w_h\|_{0,\Omega},\ \ \forall\,\varphi\in H^4_0(\Omega),\ w_h\in B^{3+}_{h0}+H^2_0(\Omega).
\end{equation}
\end{lemma}
\begin{proof}
Given $e\in\mathcal{E}_h$, by the definition of $B_{h0}^{3+}$, $\fint_{e}p_e\llbracket\partial_{\mathbf{n}_e}w_h\rrbracket_e=0,\ p_e\in P_1(e)$; for the tangential direction, $\fint_{e}p_e\llbracket\partial_{\tau_e}w_h\rrbracket_e=(p_e(L_e)\llbracket w_h\rrbracket_e(L_e)-p_e(R_e)\llbracket w_h\rrbracket_e(R_e))-\fint_{e}\partial_{\tau_e}p_e\llbracket w_h\rrbracket_e=0$. Namely, 
\begin{equation}\label{eq:momentc3+}
\fint_{e}p_e\llbracket\nabla w_h\rrbracket_e=\undertilde{0},\ \forall\, p_e\in P_1(e),\ \ e\in\mathcal{E}_h.
\end{equation}
Therefore, \eqref{eq:res1} follows by standard techniques.

Now define $\Pi_h^2$ the nodal interpolation to $\mathcal{L}_{h0}^2$ by
$$
(\Pi_h^2w)(a)=w(a),\ \forall\,a\in\mathcal{X}_h^i;\quad \fint_e(\Pi_h^2w)=\fint_ew,\ \forall\,e\in\mathcal{E}_h^i.
$$
It is easy to verify the operator is well-defined. Moreover,
\begin{equation}\label{eq:orthpi2}
\fint_T\undertilde{c}\cdot\nabla (w-\Pi_h^2w)=0,\ \ \forall\,\undertilde{c}\in\mathbb{R}^2\ \mbox{and}\ T\in\mathcal{T}_h, \ \mbox{provided}\ w\in H^2_0(\Omega)+B^{3+}_{h0}.
\end{equation}
By Green's formula,
\begin{equation}
(\Delta^2u,\Pi_h^2 w_h)=-(\nabla\Delta u,\nabla \Pi_h^2 w_h),
\end{equation}
Therefore, 
$$
\displaystyle\mathcal{R}_h^2(\varphi,w_h)=(\nabla \Delta u,\nabla_h(w_h-\Pi_h^2w_h))+(\Delta^2u,w_h-\Pi_h^2w_h):=I_1+I_2.
$$
By \eqref{eq:orthpi2},
$$
I_1=\inf_{\undertilde{c}\in(\mathbb{P}_h^0)^2}\left(\left[\nabla\Delta u-\undertilde{c}\right],\nabla_h(\Pi_h^2w_h-w_h)\right)\leqslant Ch^2|u|_{4,\Omega}\|\nabla_h^2w_h\|_{0,\Omega}.
$$
Further, 
$$
I_2\leqslant Ch^2|u|_{4,\Omega}\|\nabla_h^2w_h\|_{0,\Omega}.
$$
Summing all above proves \eqref{eq:res2}. 
\end{proof}

\paragraph{\bf Proof of Theorem \ref{thm:ratep3}} 
The proof is mainly along the line of \cite{Shi.Z1990} with some technical modifications. By Strang lemma,
$$
\|\nabla_h^2(u-u_h)\|_{0,\Omega}\cequiv \inf_{v_h\in B^3_{h0}}\|\nabla_h^2(u-v_h)\|_{0,\Omega}+\sup_{v_h\in B^3_{h0}\setminus\{\mathbf{0}\}}\frac{(\nabla^2u,\nabla^2_hv_h)-(f,v_h)}{\|\nabla_h^2v_h\|_{0,\Omega}}.
$$
The approximation error estimate follows by Lemma \ref{lem:approxb3}. By Lemma \ref{lem:res}, 
$$
(\nabla^2u,\nabla^2_hv_h)-(f,v_h)=(\nabla^2u,\nabla^2_hv_h)-(\Delta^2u,v_h)=\mathcal{R}_h(u,v_h)\leqslant Ch^2|u|_{4,\Omega}\|\nabla_h^2v_h\|_{0,\Omega}.
$$
The proof of \eqref{eq:h2error3} is completed.

Now we turn to the proof of \eqref{eq:h1error3}. By Lemma \ref{lem:approxb3}, there exists $u^\Pi_h\in B^3_{h0}$, such that 
$\|\nabla_h^j(u-u^\Pi_h)\|_{0,\Omega}\leqslant Ch^{4-j}|u|_{4,\Omega}$, $j=1,2$, and  thus $\|\nabla_h^2(u_h^\Pi-u_h)\|_{0,\Omega}\leqslant Ch^2|u|_{4,\Omega}$. Denote by $\Pi_h^1$ the nodal interpolation onto $\mathcal{L}^1_{h0}$, then $\Pi_h^1(u^\Pi_h-u_h)\in H^1_0(\Omega)$. Set $\varphi\in H^3(\Omega)\cap H^2_0(\Omega)$ such that 
$$
(\nabla^2\varphi,\nabla^2 v)=(\nabla\Pi_h^1(u^\Pi_h-u_h),\nabla v),\quad\forall\,v\in H^2_0(\Omega),
$$
then when $\Omega$ is convex, $\|\varphi\|_{3,\Omega}\cequiv \|\Pi_h^1(u^\Pi_h-u_h)\|_{1,\Omega}$. By Green's formula, 
\begin{multline*}\label{eq:shi23'}
\|\nabla\Pi_h^1(u^\Pi_h-u_h)\|_{0,\Omega}^2=-(\nabla\Delta\varphi,\nabla \Pi_h^1(u_h^\Pi-u_h))=-(\nabla\Delta\varphi,\nabla\Pi_h^1(u_h^\Pi-u))-(\nabla\Delta \varphi,\nabla\Pi_h^1(u-u_h))
\\
=(\nabla\Delta\varphi\cdot\nabla ({\rm Id}-\Pi_h^1)(u^\Pi_h-u_h))-(\nabla\Delta \varphi\cdot\nabla(u^\Pi_h-u))-(\nabla\Delta \varphi\cdot\nabla(u-u_h)):=I_1+I_2+I_3.
\end{multline*}
Further, choose $\varphi_h^\Pi$ to be an approximation of $\varphi$ by Lemma \ref{lem:approxb3}, and
\begin{multline*}
I_3=(\nabla^2\varphi,\nabla_h^2(u-u_h))+\mathcal{R}_h^1(\varphi,u-u_h)=-(\nabla_h^2(\varphi-\varphi_h^\Pi),\nabla_h^2(u-u_h))-(\nabla_h^2\varphi_h^\Pi,\nabla_h^2(u-u_h))+\mathcal{R}_h^1(\varphi,u-u_h)
\\
=-(\nabla_h^2(\varphi-\varphi_h^\Pi),\nabla_h^2(u-u_h))-\mathcal{R}_h(u,\varphi-\varphi_h^\Pi)+\mathcal{R}_h^1(\varphi,u-u_h)
\end{multline*}
Therefore $\|\nabla\Pi_h^1(u^\Pi_h-u_h)\|_{0,\Omega}^2\leqslant Ch^3|\varphi|_{3,\Omega}|u|_{4,\Omega},$ and $\|\nabla\Pi_h^1(u_h^\Pi-u_h)\|_{0,\Omega}\leqslant Ch^3|u|_{4,\Omega}$. Finally
\begin{multline*}
\|\nabla_h(u-u_h)\|_{0,\Omega}\leqslant \|\nabla_h(u-u_h^\Pi)\|_{0,\Omega}+\|\nabla_h(u_h^\Pi-u_h)\|_{0,\Omega}
\\
\leqslant  
\|\nabla_h(u-u_h^\Pi)\|_{0,\Omega}+\|\nabla_h[(u_h^\Pi-u_h)-\Pi_h^1(u_h^\Pi-u_h)]\|_{0,\Omega}+\|\Pi_h^1(u_h^\Pi-u_h)]\|_{0,\Omega}\leqslant Ch^3\|\nabla^4u\|_{0,\Omega}.
\end{multline*}
The proof is completed.
\qed

\subsection{On the implementation and solver} The lemma below follows from Theorem \ref{thm:exactb3} and Lemma \ref{eq:imp3}. Namely, \eqref{eq:bhvpp3} can be decomposed to three subsystems to implement and to solve. 

\begin{lemma}
Let $u_h^*$ be obtained by following procedure: 
\begin{enumerate}
\item find $r_h\in A^3_{h0}$, such that 
$$
(\nabla_hr_h,\nabla_hs_h)=(f,s_h),\quad\forall\,s_h\in A_{h0}^3;
$$
\item find $(\uphi{}_h,p_h)\in \uG{}^2_{h0}\times\mathbb{P}^1_{h0}$, such that 
\begin{equation*}
\left\{
\begin{array}{rll}
(\nabla_h\uphi{}_h,\nabla_h\upsi{}_h)+(p_h,\rot_h\upsi{}_h)&=(\nabla_hr_h,\upsi{}_h)&\forall\,\upsi{}_h\in\uG{}^2_{h0}
\\
(q_h,\rot_h\uphi{}_h)&=0&\forall\,q_h\in \mathbb{P}^1_{h0};
\end{array}
\right.
\end{equation*}
\item find $u_h^*\in A^3_{h0}$, such that 
$$
(\nabla_hu_h^*,\nabla_hv_h^*)=(\uphi{}_h,\nabla_hv_h^*),\quad\forall v_h^*\in A^3_{h0}.
$$
\end{enumerate}
Let $u_h$ be the solution of \eqref{eq:bhvpp3}. Then $u_h^*=u_h$.
\end{lemma}
\begin{remark}
The space $B_{h0}^3$ admits a set of basis functions whose supports are located within single vertex-patches, and the finite element scheme can be implemented by writing these basis functions out. However, the way of implementation as Lemma \ref{eq:imp3} is recommended, as 
\begin{itemize} 
\item[---] it involves Poisson equations and Stokes problems only, which can be solved in effective and standard ways;
\item[---] only restriction on the evaluation, rather than the derivatives, are imposed on the finite element functions involved. 
\end{itemize}
\end{remark}

\begin{remark}
Evidently, $B^3_{h0}=\{w_h\in A^3_{h0}:\fint_ep_e\llt\partial_{\bf n}w_h\rrt=0,\ \forall\, p_e\in \tilde{P}_1(e), \forall\, e\in\mathcal{E}_h^i\}$, where $\tilde{P}_1(e)$ denotes the first-degree homogeneous polynomial on $e$ with vanishing integration. The restriction that $\partial_{\bf n}w_h$ is moment-continuous across interior edges can be imposed by a penalty formulation with the type of saddle point problem. 
\end{remark}

\section{An optimal piecewise quartic finite element method for biharmonic equation}
\label{sec:quartic}

In this section, we assume on the triangulation that
\begin{quote}
{\bf Assumption CF}:\ \ \ $\rot_h\uG{}^3_{h0}=\mathbb{P}^2_{h0}.$
\end{quote}
Regarding Remark \ref{rem:cfinfsup}, the assumption is quite mild.

Now define
\begin{multline*}
B_h^{4}:=\{v\in L^2(\Omega):v|_T\in P_4(T), v\ \mbox{is\ continuous\ at}\ a\in\mathcal{X}_h;
\\ 
\int_ep_e\llt v\rrt=0,\ \forall\, p_e\in P_1(e),\ \mbox{and}\  \int_eq_e\llt\partial_{\bf n}v\rrt=0,\ \forall\,q_e\in P_2(e), \ \forall\, e\in\mathcal{E}_h^i\},
\end{multline*}
and, corresponding to the boundary condition of $H^2_0(\Omega)$, define
$$
B_{h0}^{4}:=\{v\in B_h^{4}: v(a)=0,\ a\in\mathcal{X}_h^b;\ \int_ep_ev=0,\ p_e\in P_1(e),\ \mbox{and}\  \int_eq_e\partial_{\bf n}v=0,\  \forall\,q_e\in P_2(e),\ \forall\, e\in\mathcal{E}_h^b\}.
$$
We consider the discretization: find $u_h\in B^4_{h0}$, such that 
\begin{equation}\label{eq:bhvpp4}
(\nabla_h^2u_h,\nabla^2_hv_h)=(f,v_h),\quad\forall\,v_h\in B^4_{h0}.
\end{equation}
The main result of this section is the theorem below. Denote
\begin{equation}
\mathcal{C}_h^{\rm CF}:=\inf_{q_h\in \mathbb{P}_{h0}^2\setminus\{\mathbf{0}\}}\sup_{\uv{}_h\in\uG{}^3_{h0}\setminus\{\mathbf{0}\}}\frac{(\rot_h\uv{}_h,q_h)}{\|\nabla_h\uv{}_h\|_{0,\Omega}\|q_h\|_{0,\Omega}}.
\end{equation}
\begin{theorem}\label{thm:ratep4}
Let $u$ and $u_h$ be the solutions of \eqref{eq:bhvp} and \eqref{eq:bhvpp4}, respectively. If $u\in H^5(\Omega)$, then, with $C$ a generic constant depending on $\Omega$ and the regularity of the grid only,
\begin{equation}
\|\nabla_h^2(u-u_h)\|_{0,\Omega}\leqslant \frac{Ch^3}{\mathcal{C}_h^{\rm CF}}|u|_{5,\Omega}.
\end{equation}
\end{theorem}
The proof of Theorem \ref{thm:ratep4} is just the same as that of Theorem \ref{thm:ratep3}. Firstly we list some key lemmas and omit the detailed proofs.

\begin{lemma}
A discretized Stokes complex holds as below:
\begin{equation}
\begin{array}{ccccccccc}
0 & \longrightarrow & B_{h0}^{4+} & \xrightarrow{\bs{\mrm{\nabla}}_h} & \uG{}_{h0}^{3+} & \xrightarrow{\mrm{rot}_h} & \mathbb{P}_{h0}^2  & \longrightarrow & 0.
\end{array}
\end{equation}
\end{lemma}

\begin{lemma}
$B_{h0}^4=\{w_h\in A^4_{h0}:\nabla_hw_h\in \uG{}^4_{h0}\}$.
\end{lemma}

\begin{remark}\label{rem:strB4}
Evidently, $B_{h0}^4=\{w_h\in B_{h0}^{4+}:w_h|_T\in P_4(T),\ \forall\,T\in\mathcal{T}_h\}$, and $\displaystyle B_{h0}^{4}=\{v\in A^4_{h0}, \ \int_ep_e\llt\partial_{\bf n}v\rrt=0,\ \forall\, p_e\in \tilde{P}_2(e), \forall\,e\in\mathcal{E}_h\},$
where $\tilde{P}_2(e)$ consists of polynomials in $P_2(e)$ with vanishing integration.
\end{remark}

\begin{theorem}
A discretized Stokes complex holds as below:
\begin{equation}
\begin{array}{ccccccccc}
0 & \longrightarrow & B_{h0}^{4} & \xrightarrow{\bs{\mrm{\nabla}}_h} & \uG{}_{h0}^{3} & \xrightarrow{\mrm{rot}_h} & \mathbb{P}_{h0}^2  & \longrightarrow & 0.
\end{array}
\end{equation}
\end{theorem}

\begin{lemma}\label{lem:approxb4}
There exists a constant $C$, such that for any $w\in H^2_0(\Omega)\cap H^4(\Omega)$, there exists $w_h\in B^3_{h0}$ and 
$$
\|\nabla_h^2(w-w_h)\|_{0,\Omega}\leqslant \frac{Ch^3}{\mathcal{C}_h^{\rm CF}}|w|_{5,\Omega};
$$
moreover, if $\Omega$ is convex, then $w_h$ can be chosen such that 
$$
\|\nabla_h(w-w_h)\|_{0,\Omega}\leqslant \frac{Ch^4}{\mathcal{C}_h^{\rm CF}}|w|_{5,\Omega}.
$$
\end{lemma}

\begin{lemma}\label{lem:resquartic}
Recall $\mathcal{R}_h^1$ and $\mathcal{R}_h^2$ defined in \eqref{eq:res1def} and \eqref{eq:res2def}. There is a constant $C$, such that 
\begin{equation}\label{eq:res1q}
\displaystyle\mathcal{R}_h^1(\varphi,w_h)\leqslant Ch^k|\varphi|_{2+k,\Omega}\|\nabla_h^2w_h\|_{0,\Omega},\ \ \forall\,\varphi\in H^2_0(\Omega)\cap H^{k+2}(\Omega),\ w_h\in B^{4+}_{h0}+H^2_0(\Omega),\ k=1,2,3.
\end{equation}
\begin{equation}\label{eq:res2q}
\displaystyle\mathcal{R}_h^2(\varphi,w_h)\leqslant Ch^k|\varphi|_{k+2,\Omega}\|\nabla_h^2w_h\|_{0,\Omega},\ \ \forall\,\varphi\in H_2^0(\Omega)\cap H^{k+2}(\Omega),\ w_h\in B^{4+}_{h0}+H^2_0(\Omega),\ k=2,3.
\end{equation}
\end{lemma}
\begin{proof}
Firstly, we can obtain for $w_h\in B^{4+}_{h0}$ that
\begin{equation}\label{eq:momentc4+}
\fint_{e}p_e\llbracket\nabla w_h\rrbracket_e=\undertilde{0},\ \forall\, p_e\in P_2(e),\ \ e\in\mathcal{E}_h,
\end{equation}

Secondly, define $\Pi_h^3$ the nodal interpolation to $\mathcal{L}_{h0}^3$ by
$$
(\Pi_h^3w)(a)=w(a),\ \forall\,a\in\mathcal{X}_h^i;\  \fint_ep_e(\Pi_h^3w)=\fint_ep_ew,\ \forall\,e\in\mathcal{E}_h^i,\ p_e\in P_1(e);\ \fint_T\Pi_h^3w=\fint_T w,\ \forall\,T\in\mathcal{T}_h.
$$
It is easy to verify the operator is well-defined. Moreover,
\begin{equation}\label{eq:orthpi3}
\fint_T\undertilde{c}\cdot\nabla (w-\Pi_h^2w)=0,\ \ \forall\,\undertilde{c}\in\uP{}_1(T)\ \mbox{and}\ T\in\mathcal{T}_h, \ \mbox{provided}\ w\in B^{4+}_{h0}.
\end{equation}

The remaining of the proof is the same as that of Lemma \ref{lem:res}.
\end{proof}

\begin{remark}
As $\dim(B^4_{h0})=\dim(\{\uw{}_h\in\uG{}^3_{h0}:\rot_h\uw{}_h=0\})=3\#(\mathcal{X}^i_{h})+2\#(\mathcal{E}_h^i)-3$, $B^4_{h0}$ on a triangulation with no interior vertex and two interior edges is 1-dimensional. Thus the supports of the basis of $B^4_{h0}$ can be with more possibilities than that of the basis of $B^3_{h0}$, and we do not seek to present a detailed description in the present paper. 
\end{remark}

\subsection{On the implementation and optimal solver}
Again, the discretization scheme \eqref{eq:bhvpp4} can be implemented by solving two Poisson systems and a Stokes system. 

\begin{lemma}
Let $u_h^*$ be obtained by following procedure: 
\begin{enumerate}
\item find $r_h\in A^4_{h0}$, such that 
$$
(\nabla_hr_h,\nabla_hs_h)=(f,s_h),\quad\forall\,s_h\in A^4_{h0};
$$
\item find $(\uphi{}_h,p_h)\in \uG{}^3_{h0}\times\mathbb{P}^2_{h0}$, such that 
\begin{equation*}
\left\{
\begin{array}{ll}
(\nabla_h\uphi{}_h,\nabla_h\upsi{}_h)+(p_h,\rot_h\upsi{}_h)=(\nabla_hr_h,\upsi{}_h)&\forall\,\upsi{}_h\in\uG{}^3_{h0}
\\
(q_h,\rot_h\uphi{}_h)=0&\forall\,q_h\in \mathbb{P}^2_{h0};
\end{array}
\right.
\end{equation*}
\item find $u_h^*\in A^4_{h0}$, such that 
$$
(\nabla_hu_h^*,\nabla_hv_h^*)=(\uphi{}_h,\nabla_hv_h^*),\quad\forall v_h^*\in A^4_{h0}.
$$
\end{enumerate}
Let $u_h$ be the solution of \eqref{eq:bhvpp4}. Then $u_h^*=u_h$.
\end{lemma}

\section{Concluding remarks}
\label{sec:conc}

In this paper, piecewise cubic and piecewise quartic element spaces for biharmonic equation are constructed with optimal convergence rate with respect to energy norm. The schemes can be implemented and the generated systems can be optimally solved in quite a friendly way. Together with the optimal quadratic element space (Morley element) and optimal high-degree (not smaller than 5) $\mathcal{C}^1$ spaces, these form a complete answer to the question whether and how arbitrary order optimal finite element schemes can be constructed for planar biharmonic equation. The schemes may be expected to find practical applications in the computation of both source and eigenvalue problems. 

~\\

A basic tool is the discretized Stokes complexes in this paper, and the role of the Stokes complex that it is a tool for studying biharmonic equation is illustrated. The optimality of the schemes constructed is relevant to the stability of the corresponding Stokes element pairs; the Fortin-Soulie pair is stable on general grids, and there remains no more than a tiny gap if the Crouziex-Falk pair does. Besides, similar to \cite{Zhang.S2017axv} where a discretized Stokes complex on quadrilateral grids (\cite{Zhang.S2016nm}) is used for aid, in this paper, an auxiliary discretized Stokes complex with locally defined finite element spaces is constructed for the discretized Stokes complex with spaces defined globally. These routine techniques for deriving exact sequences may find more usages in future, especially usage in counting the dimension of certain spaces. Moreover, in contrast to general practice, the convergence analysis of finite element schemes is illustrated to be useful for the proof of approximation for finite element space to certain Sobolev spaces in this paper. This is why, though the piecewise cubic element space admits a space of locally supported basis functions and a procedure how to establish such a basis is given in the paper, we do not seek to establish the theory based upon interpolation. In the future, how the basis functions can be used for the implementation of the scheme, and how the approximation of the space can be estimated in a more direct way will be discussed. 
~\\

The schemes constructed in the present paper can be viewed as a generalization of the Morley element to higher-degree polynomials. Actually, they three, namely the Morley element space, the spaces $B^3_h$ and the spaces $B_h^4$, belong to a same family which reads in two dimension:
\begin{multline}
B_h^k:=\{w_h\in L^2(\Omega):w_h(a)\ \mbox{is\ continuous\ at}\ a\in\mathcal{X}_h;
\\ 
\fint_e\llt w_h\rrt p_e=0, \ \forall\,p_e\in P_{k-3}(e), \fint_eq_e\llt\partial_{\bf n}w_h\rrt=0,\ \forall\,p_e\in P_{k-2}(e),\ \forall\,   e\in\mathcal{E}_h^i\}.
\end{multline}
The space $B_{h0}^k\subset B_h^k$ can be defined corresponding to the boundary condition of $H^2_0(\Omega)$. It has been proved that $B_{h(0)}^k$ is an optimally consistent finite element space for biharmonic equation for $k=2,3$ for arbitrary triangulations and for $k=4$ for most reasonable triangulations. Can the family work optimally with arbitrary $k\geqslant 2$ and can it be generalized to higher dimension and even higher orders? This can be a question of interests and be studied in future. Structural properties relevant to the spaces can be studied together. 
~\\

By Remarks \ref{rem:strB3dg} and \ref{rem:strB4}, $B^3_{h0}$ and $B^4_{h0}$ can be viewed as restricted $A^3_{h0}$ and $A^4_{h0}$  with respect to continuity  on derivatives. This way the schemes can be implemented like discontinuous Galerkin methods, where interior penalties are replaced by continuity restrictions. This hints the possible asymptotic connection between the solutions of these scheme and the solutions of discontinuous Galerkin schemes, which will be studied in future.


\end{document}